\documentclass[10pt,a4paper]{article}
\usepackage[utf8]{inputenc}
\usepackage[T1]{fontenc}
\usepackage{amsmath}
\usepackage{amssymb}
\usepackage{graphicx}

\usepackage{accents} 
\usepackage{hyperref} 
\usepackage{braket}
\usepackage{amsthm}
\usepackage{algorithm}
\usepackage{enumitem} 
\usepackage{url} 

\newcommand{\ubar}[1]{\underaccent{\bar}{#1}}
\DeclareMathOperator*{\argmin}{arg\,min}

\theoremstyle{plain}
\newtheorem{theorem}{Theorem}[section]
\newtheorem{proposition}[theorem]{Proposition}

\theoremstyle{definition}

\newtheorem{example}[theorem]{Example}

\theoremstyle{remark}
\newtheorem{remark}[theorem]{Remark}

\begin{document}

\title{Bidirectional Optimisation for Load Shaping within Coupled Microgrids}

\author{Philipp Sauerteig\thanks{Optimization-based Control group, Institute of Mathematics, Technische Universit\"at Ilmenau, Ilmenau, Germany, philipp.sauerteig@tu-ilmenau.de}}

\maketitle

\begin{abstract}%
	We address the problem of load shaping within a network of coupled microgrids (MGs) in a bilevel optimisation framework.\ %
	To this end, we consider the charging/discharging rates of residential energy storage devices within each MG on the lower level and the power exchange among neighbouring MGs on the upper level as optimisation variables.\ %
	We improve a previously developed model such that the maximal amount of exchanged power does not depend on the power demand, thus, increasing the flexibility within the network, and adapt the corresponding bidirectional optimisation scheme accordingly.\ %
	For efficiency, standard distributed optimisation routines are used for the optimisation on the lower level; the power exchange problem on the upper level is replaced by parallelisable small-scale quadratic programmings.\ %
	We prove global convergence of the optimisation scheme and illustrate the potential of the approach in a numerical case study based on real-world data.\ %
\end{abstract}

\section{Introduction}
Over the last decades more and more processes throughout the entire power grid have been automated.\ %
Behind these automations are complex optimisation algorithms that ensure proper operation of the so-called \emph{smart grid}~\cite{RuizColm14,Farh10,MoldBakk10}.\ %
Traditionally, power grids are hierarchically structured with power flowing from the top to the bottom layer.\ %
With the increasing number of renewable energy sources on a residential level and the resulting necessity for local storages, the distribution grid becomes more active and offers new optimisation potential~\cite{ZafaMahm18,GrijTari11}.\ %
Typically, a collection of such residential energy systems is called a \emph{microgrid} (MG).\ %
There are several possible objectives that need to be addressed when optimising MGs.\ %
One of the probably most important ones is peak shaving~\cite{OudaCher07,WangWang13} or, more generally, load shaping~\cite{PateErdi15}.\ %
Recent research articles suggest to use distributed optimisation algorithms for optimal control of MGs, see, e.g.~\cite{MaknQu14,WortKell15,LiuTan18,BrauFaul18} or the survey article~\cite{MolzDoer17}.\ %
An extension of the optimisation of a single MG is to consider interconnected MGs~\cite{SampHofm21,TianXiao16}.\ %
Instead of having one layer, one then has to consider a grid hierarchy and use tailored algorithms to handle the communication between the MGs and the superordinate organising unit.\ %
In~\cite{MurrEnge18}, for instance, the authors propose to use an augmented Lagrangian alternating direction inexact Newton (ALADIN) method to solve a hierarchical mixed-integer optimisation.\ %
For two-layer optimisation problems there exists bilevel optimisation methods, of which an overview is given in~\cite{ColsMarc07}.\ %
In~\cite{BahrMogh16} the authors use Karush-Kuhn-Tucker (KKT) conditions and dual theory to translate the problem into a linear single-level problem.\ %
Stochastical bilevel for the bidding strategy of a power plant is considered in~\cite{KardSimo16}.\ %
Multiobjective bilevel optimisation problems have been addressed, e.g.\ in~\cite{Eich10}.\ %
State of the art to incorporate prediction based on, e.g.\ weather forecasts, is model predictive control (MPC).\ %
In the context of MG optimisation see, e.g.~\cite{PariRiko14} for a mixed-integer linear problem or the review article~\cite{Scat09} for distributed hierarchical MPC.\ %

In~\cite{BrauSaue19}, the authors consider a network of partially coupled MGs, which are equipped with residential energy storage units (batteries), that can be coordinated in order to reduce peaks in the power demand.\ %
Additionally, the MGs are able to exchange power with their neighbours to improve the overall peak-shaving performance yielding a bilevel optimisation problem.\ %
For solving this problem, a bidirectional optimisation scheme was proposed in~\cite{BaumGrun19}.\ %
To this end, the exchanged power is taken into account as additional demand/supply in the lower level optimisation problem, thus, possibly changing the optimal battery control.\ %
The updated charging profile in turn comes along with a new aggregated power demand within each MG changing the parameters in the power exchange problem on the upper level and so on.\ %
In conclusion, the bilevel problem is solved iteratively in a negotiation-like process.\ %
The model for the power exchange, however, has some slight disadvantages.\ %
In particular, the maximal amount of power that can be exchanged is a fraction of the aggregated power demand within the respective MG.\ %
Thus, if the demand is balanced (approximately zero) no power can be exchanged.\ %
In this paper, we improve the model to circumvent this problem and adapt the proposed bidirectional optimisation scheme accordingly.\ %
The benefit of the new formulation is twofold:\ %
(1) from a practical point of view, it enables us to exploit more flexibility and, thus, reduce the overall costs and (2) from a theoretical point of view,\ %
the resulting optimisation problem is convex, which allows us to adapt the techniques used in~\cite{BrauGrue16a} to prove global convergence.\ %
Furthermore, we demonstrate the efficiency of the proposed setting by incorporating it within an MPC framework showing that it is suited to reduce the overall load shaping costs based on a novel real-world data set.\ %

The remainder of this paper is structured as follows.\ %
In Section~\ref{sec:model}, we discuss the underlying model and formulate the optimisation problem.\ %
In Section~\ref{sec:bidir_optim}, we adapt the bidirectional optimisation scheme proposed in~\cite{BaumGrun19} to the new problem formulation.\ %
Section~\ref{sec:numerics} is dedicated to a numerical cased study using real-world data before we conclude in Section~\ref{sec:conclusions}.\ %

Throughout this paper, we use the notation $[\ell:m]$ to describe the set of all integers $\{\ell, \ell+1, \ldots, m\}$ from~$\ell$ to~$m$ for any integers $\ell, m \in \mathbb{Z}$ with $\ell \leq m$.\ %

\section{Model and Problem Formulation}\label{sec:model}
We study two levels of the grid hierarchy:\ %
the lower level is a collection of residential energy systems forming a \emph{microgrid} (MG) while the upper level consists of \emph{coupled microgrids}.\ %
On the lower level, energy storage devices are controlled such that peaks in the aggregated power demand are reduced.\ %
To improve the peak shaving even further, power exchange among neighbouring MGs is optimised on the upper level.\ %

\subsection{Peak Shaving within a Single Microgrid}\label{sec:model:single}
The underlying model has been developed in~\cite{WortKell14} and extended, e.g.\ in~\cite{BrauFaul18,GrunSaue19}.\ %
We consider a network of $\mathcal{I}$, $\mathcal{I} \in \mathbb{N}$, residential energy systems.\ %
At time instant $n \in \mathbb{N}_0$, each system~$i$, $i \in [1:\mathcal{I}]$, comprises its load~$\ell_i(n)$ [kW] as well as some energy generation~$g_i(n)$ [kW] and storage device (battery) with a dynamically changing state of charge (SoC) $x_i(n)$ [kWh].\ %
Moreover, load and generation are combined to the net consumption $w_i = \ell_i - g_i$.\ %
By charging/discharging $u_i(n) = (u_i^+(n), u_i^-(n))^\top$ [kW] the battery, system~$i$ is able to control its power demand~$z_i(n)$ [kW] as described by the discrete-time system dynamics
\begin{subequations}\label{subeq:system_dynamics}
	\begin{align}
		x_i(n+1) \; & = \; \alpha_i x_i(n) + T(\beta_i u_i^+(n) + u_i^-(n)) \label{eq:dyn_soc} \\
		z_i(n) \; & = \; w_i(n) + u_i^+(n) + \gamma_i u_i^-(n). \label{eq:dyn_demand}
	\end{align}
\end{subequations}
Here, the parameters $\alpha_i, \beta_i, \gamma_i \in (0,1]$ describe efficiencies with respect to self discharge, charging, and discharging, respectively.\ %
The length of a time step is denoted by $T > 0$ [h].\ %
The battery dynamics are subject to the constraints
\begin{subequations}\label{subeq:battery_constraints}
	\begin{eqnarray}
		0 \; \leq \; & x_i(n) & \leq \; C_i, \\
		\ubar{u}_i \; \leq \; & u_i^-(n) & \leq \; 0, \\
		0 \; \leq \; & u_i^+(n) & \leq \; \bar{u}_i, \\
		0 \; \leq \; & \frac{u_i^-(n)}{\ubar{u}_i} + \frac{u_i^+(n)}{\bar{u}_i} \; & \leq \; 1, \label{eq:constraint:charge_and_discharge}
	\end{eqnarray}
\end{subequations}
where $C_i \geq 0$ [kWh] denotes the battery capacity and $\ubar{u}_i \leq 0$ [kW] and $\bar{u}_i \geq 0$ [kW] represent maximal discharging and charging rates, respectively.\ %
Here, constraint~\eqref{eq:constraint:charge_and_discharge} ensures that those bounds are also satisfied if batteries are both charged and discharged during one time step.\ %
In particular for large time steps~$T$, it might be optimal to charge and discharge within one step in order to dissipate superfluous energy due to conversion losses $1 - \beta_i$ and $1 - \gamma_i$.\ %

Each system is connected to the microgrid operator (MGO) as depicted in Figure~\ref{fig:micorgrid}.\ %
\begin{figure}[h]
	\centering
	\includegraphics[width=\columnwidth]{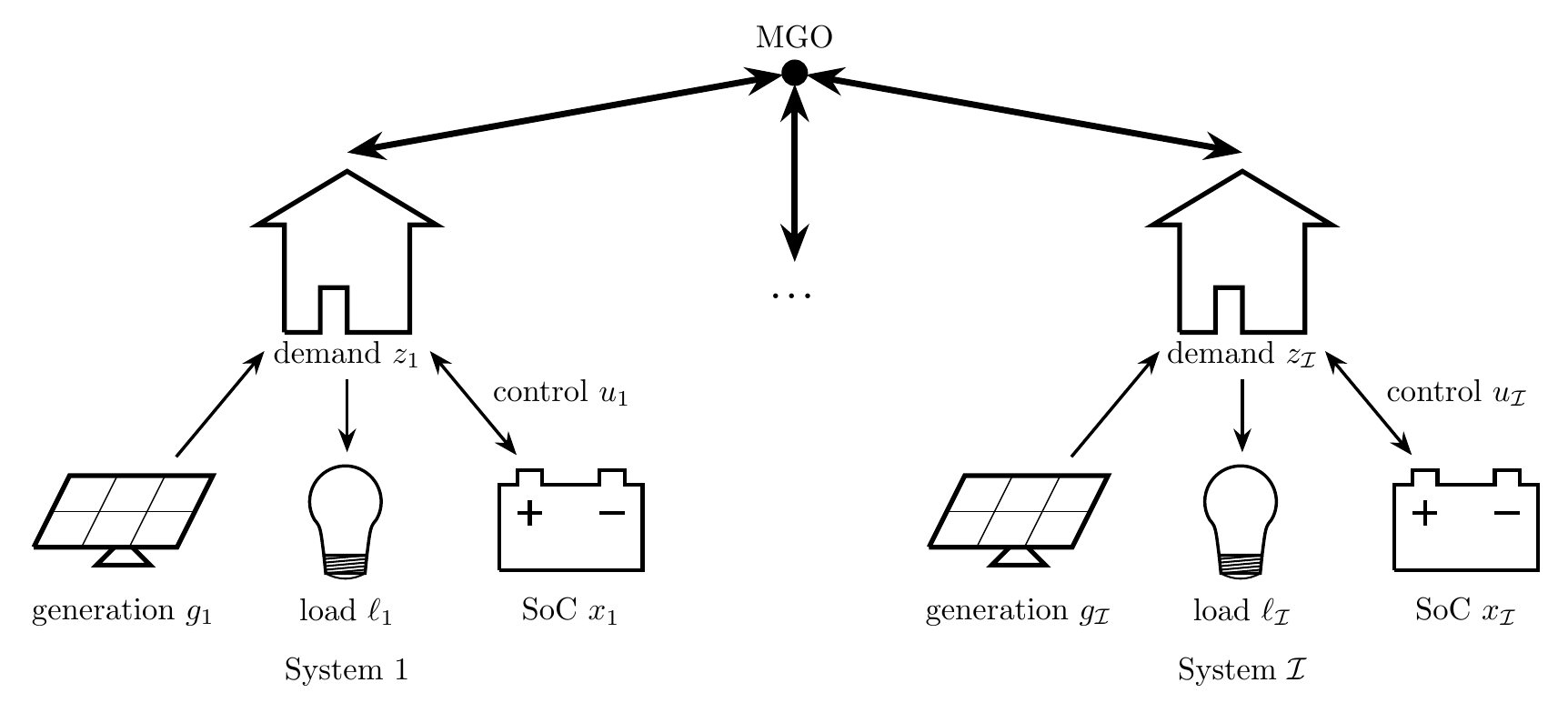}
	\caption{One microgrid consisting of~$\mathcal{I}$ residential energy systems each incorporating load, generation, and storage.\ %
		Each system is connected to the microgrid operator (MGO).\ %
		Arrows indicate the possible direction of power flow.}
	\label{fig:micorgrid}
\end{figure}
Typically, the aggregated power demand profile $\bar{z} = \sum_{i=1}^\mathcal{I} z_i$, which has to be compensated by the MGO is volatile, e.g.\ due to time-dependent consumption and weather-dependent generation.\ %
Consequently, one goal from a MGO's point of view is peak shaving or, more generally, load shaping, i.e., the MGO is interested in a \emph{nice} (e.g.\ constant) demand profile~$\zeta$ [kW].\ %
This can be achieved by (dis-)charging the residential batteries such that 
\begin{align}
	\sum_{n = 0}^\infty \left( \bar{z}(n) - \zeta(n) \right)^2 \notag
\end{align}
is minimised.\ %
Keep in mind, that the future demand~$z_i(n)$, depends on the future net consumption~$w_i(n)$, see~\eqref{eq:dyn_demand}, which is unknown at the current time instant $k \in \mathbb{N}_0$.\ %
However, we assume the future net consumption $w_i = (w_i(k), \ldots, w_i(k+N-1))^\top$ to be predictable on a sufficiently small time window of~$N$, $N \in \mathbb{N}_{\geq 2}$, time steps.\ %
From here on, we use the notation $z_i = (z_i(k), \ldots, z_i(k+N-1))^\top \in \mathbb{R}^N$ to denote the future power demand of system~$i$ over the prediction horizon~$N$, similar for other variables.\ %
In conclusion, at time instant~$k$, we aim to minimise
\begin{align}
	\sum_{n = k}^{k+N-1} \left( \bar{z}(n) - \zeta(n) \right)^2 \notag
\end{align}
subject to the system dynamics and constraints.\ %

Next we introduce some notions we will make use of in the subsequent sections.\ %
Note that the battery dynamics~\eqref{eq:dyn_soc} as well as the constraints~\eqref{subeq:battery_constraints} are linear.\ %
Once the current SoC~$\hat x_i$ has been measured, both state and control constraints can be written in the form
\begin{align}
	D_i u_i \; \leq \; d_i \quad \forall \, i \in [1:\mathcal{I}] \quad \text{or} \quad D u \; \leq \; d \notag
\end{align}
with suitable matrices $D_i$, $D$ and vectors $d_i$, $d$.\ %
Here, we stacked the controls 
\begin{align}
	u_i = (u_i(k)^\top, \ldots, u_i(k+N-1)^\top)^\top \in \mathbb{R}^{2N} \notag
\end{align}
and $u = (u_1^\top, \ldots, u_\mathcal{I}^\top)^\top \in \mathbb{R}^{2N \mathcal{I}}$.\ %
Furthermore, the demand equation~\eqref{eq:dyn_demand} can be written as
\begin{align}
	z_i \; = \; A_i u_i + b_i \quad \forall \, i \in [1:\mathcal{I}] \quad \text{or} \quad \bar{z} \; = \; A u + b \notag
\end{align}
with suitable matrices~$A_i$, $A$ and vectors~$b_i$, $b$, see also~\cite{JianSaue21}.\ %
We collect all feasible control sequences in
\begin{align}
	\mathbb{U}_i \; = \; \Set{u_i \in \mathbb{R}^{2N} | D_i u_i \leq d_i} \notag
\end{align}
and feasible demand profiles in
\begin{align}
	\mathbb{D}_i \; = \; \Set{z_i \in \mathbb{R}^N | \exists \, u_i \in \mathbb{U}_i : z_i = A_i u_i + b_i}. \notag
\end{align}

Given the predicted net consumption profiles~$w_i$ and the measured current SoCs~$\hat x_i$, the MGO is interested in solving the quadratic programming (QP)
\begin{subequations}\label{subeq:OP_lower}
	\begin{align}
		\min \quad & \left\| A u + b - \zeta \right\|_2^2 \\
		\mathrm{s.t.} \quad & D u \leq d.
	\end{align}
\end{subequations}
Solving such problems efficiently in a distributed way has been analysed, e.g.\ in~\cite{WortKell14,BrauFaul18,JianSaue21}.\ %

\subsection{Power Exchange among Coupled Microgrids}
In the previous subsection, we discussed how a MGO may achieve load shaping by manipulating residential batteries.\ %
Naturally, there may be situations (high load/generation), where the batteries do not suffice to achieve this goal.\ %
In this section, we consider a network of (partially) coupled MGs and show how power exchange among neighbouring MGs can be used to improve the overall performance.\ %

We consider~$M$, $M \in \mathbb{N}$, MGs each structured as described in Section~\ref{sec:model:single}.\ %
In particular, MG~$\kappa$ comprises the aggregated power demand~$\bar z_\kappa$, $\kappa \in [1:M]$.\ %
We add a subscript~$\kappa$ whenever necessary to distinguish among MGs.\ %
Among some of these MGs there are transmission lines along which power can be exchanged.\ %
Let $\lambda_{\kappa \nu} \geq 0$ be the line limit, i.e.\ the maximal amount of power that can be transmitted between MG~$\kappa$ and MG~$\nu$, $\kappa, \nu \in [1:M]$.\ %
We assume $\lambda \in \mathbb{R}^{M \times M}$ to be symmetric and $\lambda_{\kappa \nu} > 0$ if and only if there is a transmission line between MG~$\kappa$ and MG~$\nu$ and collect the set of transmission lines (edges of a graph) in
\begin{align}
	E \; = \; \Set{(\kappa, \nu) \in [1:M]^2 | \lambda_{\kappa \nu} > 0, \; \kappa < \nu}. \notag
\end{align}
Then, by $\lambda_{\kappa \nu} \delta_{\kappa \nu}(n)$ with $\delta_{\kappa \nu}(n) \in [0,1]$ we denote the power that is actually sent from MG~$\kappa$ to MG~$\nu$ at time instant~$n$, $(\kappa, \nu) \in E$.\ %
However, each line comprises some efficiency $\eta_{\kappa \nu} \in (0,1]$.\ %
Therefore, the power received by MG~$\nu$ from MG~$\kappa$ is given by $\lambda_{\kappa \nu} \eta_{\kappa \nu} \delta_{\kappa \nu}(n)$.\ %
Line losses are assumed to not depend on the direction of the power exchange.\ %
Hence, $\eta \in \mathbb{R}^{M \times M}$ is symmetric as well.\ %

Using the notation $\delta_{\kappa \nu} = (\delta_{\kappa \nu}(k), \ldots, \delta_{\kappa \nu}(k+N-1))^\top \in \mathbb{R}^N$, $(\kappa, \nu) \in E$, and $\bar{\mathbf{z}} = (\bar z_1^\top, \ldots, \bar z_M^\top)^\top \in \mathbb{R}^{MN}$, we formulate the overall objective function $J : \mathbb{R}^{M N} \times \mathbb{R}^{2 |E| N} \to \mathbb{R}$ as
\begin{align}
	& J(\bar{\mathbf{z}}, \delta)  = \sum_{\kappa=1}^M \left\| \zeta_\kappa  - \left[ \bar{z}_\kappa + \sum_{\nu \in \mathcal{N}(\kappa)} \lambda_{\kappa \nu} ( \delta_{\kappa \nu} - \eta_{\nu \kappa} \delta_{\nu \kappa}) \right] \right\|_2^2. \notag
\end{align}
The notion $\mathcal{N}(\kappa)$ denotes the neighbouring MGs of MG~$\kappa$, i.e.,
\begin{align}
	\mathcal{N}(\kappa) \; = \; \Set{\nu \in [1:M] | \lambda_{\kappa \nu} > 0}. \notag
\end{align}
We allow power exchange in both directions along one transmission line during one time step.\ %
In order to ensure that the line limits are not violated we introduce the constraint
\begin{align}
	\delta_{\kappa \nu}(n) + \delta_{\nu \kappa}(n) \; \leq \; 1 \notag
\end{align}
for all $(\kappa, \nu) \in E$ and $n \in \mathbb{N}_0$, similar to~\eqref{eq:constraint:charge_and_discharge}.\ %
Thus, given the aggregated power demand profiles~$\bar z_\kappa$ of all MGs $\kappa \in [1:M]$, the power exchange problem reads as
\begin{subequations}\label{subeq:OP_upper}
	\begin{align}
		\min_{\delta \in \mathbb{R}^{2 |E| N}} \quad & J(\bar{\mathbf{z}}, \delta) \\
		\mathrm{s.t.} \quad & \delta_{\kappa \nu}(n) \geq 0 
		\label{eq:delta_nonnegative} \\
		& \delta_{\kappa \nu}(n) + \delta_{\nu \kappa}(n) - 1 \leq 0 \label{eq:limit_exchange_both_directions} \\
		& \quad \forall \, (\kappa, \nu) \in E, \; n \in [k:k+N-1]. \notag 
	\end{align}
\end{subequations}
We collect all feasible exchange strategies in
\begin{align}
	\mathbb{D}^\delta \; := \; \Set{ \delta \in \mathbb{R}^{2 \left| E \right| N} | \eqref{eq:delta_nonnegative} \text{ and } \eqref{eq:limit_exchange_both_directions} \text{ hold}} \notag
\end{align}
and summarise the overall optimisation problem as
\begin{align}
	\min_{(\bar{\mathbf{z}}, \delta) \in \bar{\mathbb{D}} \times \mathbb{D}^\delta} \quad J(\bar{\mathbf{z}}, \delta). \label{eq:OP_full}
\end{align}

The exchange structure of a network consisting of four MGs is depicted in Figure~\ref{fig:exchange_structure}.\ %
\begin{figure}[h]
	\centering
	\includegraphics[width=\columnwidth]{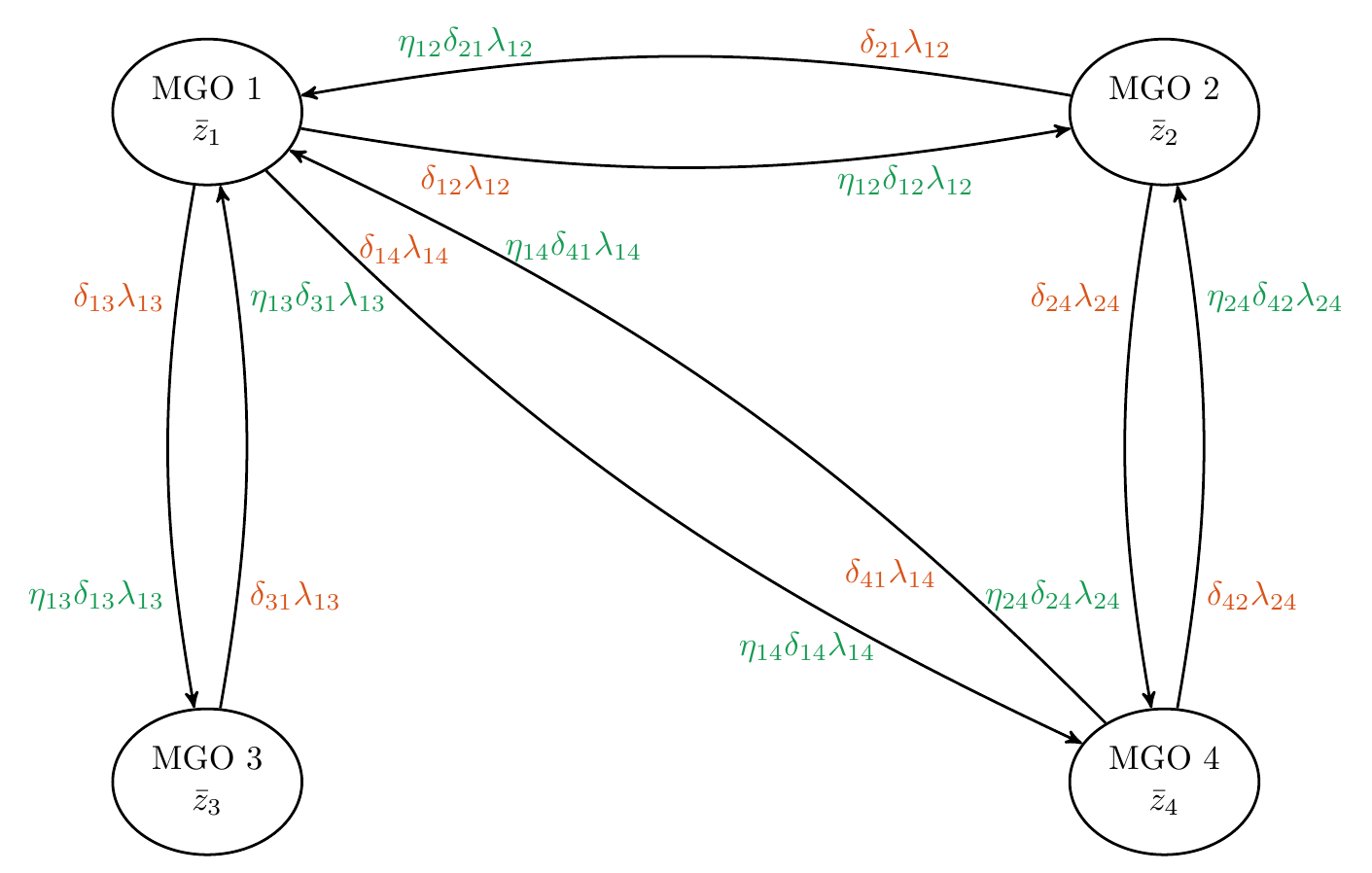}
	\caption{Power exchange structure of four partially coupled MGs.\ %
		Each MGO~$\kappa$ has his aggregated demand~$\bar{z}_\kappa$ at his disposal and may send $\delta_{\kappa \nu} \lambda_{\kappa \nu}$ units of power to MGO~$\nu$ (red).\ %
		Due to line losses, only $\eta_{\kappa \nu} \delta_{\kappa \nu} \lambda_{\kappa \nu}$ units arrive at MGO~$\nu$ (green).}
	\label{fig:exchange_structure}
\end{figure}
Note that we do not restrict the amount of power that MG~$\kappa$ is able to distribute in terms of its own demand.\ %
In particular, if 
\begin{align}
	\sum_{\nu \in \mathcal{N}(\kappa)} \lambda_{\kappa \nu} \delta_{\kappa \nu}(n) \; < \; - \bar{z}_\kappa(n), \notag
\end{align}
i.e., if the amount of power to be sent from MG~$\kappa$ to its neighbours is greater than the amount of power available in MG~$\kappa$, the DSO is assumed to compensate the difference.\ %

\begin{example}[Exchange formulation in~\cite{BaumGrun19} fails]\label{ex:old_formulation_fails}
	Consider $M = 4$ MGs without batteries, i.e., $C_{\kappa, i} = 0$ for all $i \in [1:\mathcal{I}_\kappa]$, $\kappa \in [1:M]$, and coupled as depicted in Figure~\ref{surr:fig:chained_mgs}.\ %
	For simplicity, let $\eta_{12} = \eta_{23} = \eta_{34} = 1$ and $N = 1$ as well as $\mathcal{I}_\kappa = 1$ and $\zeta_\kappa = 0$ for all $\kappa \in [1:M]$.\ %
	Let the aggregated power demands be given by $- \bar{z}_1 = \bar{z}_4 = 10$ and $\bar{z}_2 = \bar{z}_3 = 0$.\ %
	\begin{figure}[h!]
		\centering
		\includegraphics[width=\columnwidth]{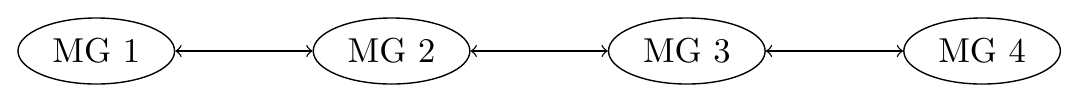}
		\caption[Four chain-like coupled MGs]%
		{Four chain-like coupled MGs.}
		\label{surr:fig:chained_mgs}
	\end{figure}
	
	In the absence of batteries we are only interested in optimising the power exchange.\ %
	Furthermore, since there are no line losses, the optimal power exchange should look like 
	\begin{itemize}
		\item MG~1 sends -10kW \emph{of its demand} to MG~2 (i.e., MG~1 sends 10kW to MG~2),
		\item MG~2 sends -10kW of its demand to MG~3, and
		\item MG~3 sends -10kW of its demand to MG~4,
	\end{itemize}
	which results in $\bar z_1 = \bar z_2 = \bar z_3 = \bar z_4 = 0$ and, thus, an optimal function value of $J^\star = 4 \cdot 0^2 = 0$.\ %
	However, in~\cite{BaumGrun19} the power exchange from MG~$\kappa$ to MG~$\nu$ is formulated as a fraction $\tilde \delta_{\nu \kappa} \bar{z}_\nu$ of the power demand $\bar{z}_\nu$ with $\tilde \delta_{\nu \kappa} \in [0,1]$.\ %
	Since $\bar{z}_2 = \bar{z}_3 = 0$, MGs~2 and~3 cannot exchange power with each other.\ %
	Hence, the optimal solution is given by 
	\begin{itemize}
		\item MG~1 sends -5kW of its demand to MG~2 and
		\item MG~4 sends 5kW of its demand to MG~3,
	\end{itemize}
	which results in $\bar z_1 = \bar z_2 = -5 $ and $\bar z_3 = \bar z_4 = 5$ and, thus, $J^\star = 4 \cdot 5^2 = 100$.\ %
	In conclusion, the formulation presented in~\cite{BaumGrun19} is not capable to exploit the full potential of the power exchange.\ %
\end{example}
In~\cite{BrauSaue19,BaumGrun19} the authors also considered power exchange among coupled MGs, however, they restricted the amount of exchanged power to a fraction of the power demand.\ %
Moreover, the authors did not allow for power exchange along both directions of a transmission line within one time step, which was encoded via complementarity constraints, which in turn made the feasible set non-convex.\ %
Our global convergence proof presented in Subsection~\ref{sec:global_conv_proof} exploits convexity of both the objective function and the feasible set and is, therefore, not applicable to the formulation in~\cite{BaumGrun19}.\ %

\section{Bidirectional Optimisation}\label{sec:bidir_optim}
Instead of solving~\eqref{eq:OP_full} centralised at once, we propose an iterative bidirectional optimisation scheme to reduce the required overhead communication.\ %
First, each MGO optimises his local power demand using the storage devices.\ %
Then, a negotiation-like process between MGOs and DSO takes place.\ %
In each step, it is determined whether power exchange may improve the overall performance.\ %
Then, based on the exchanged power, each MGO post-optimises the control of the local batteries.\ %
This procedure monotonically reduces the overall costs and ensures feasibility in each step.\ %

\subsection{Optimisation Scheme}
We propose an iterative procedure to improve the overall performance by solving~\eqref{subeq:OP_lower} and~\eqref{subeq:OP_upper} repeatedly as depicted in Figure~\ref{fig:bidir_opt_scheme}.\ %
\begin{figure}[h]
	\centering
	\includegraphics[width=\columnwidth]{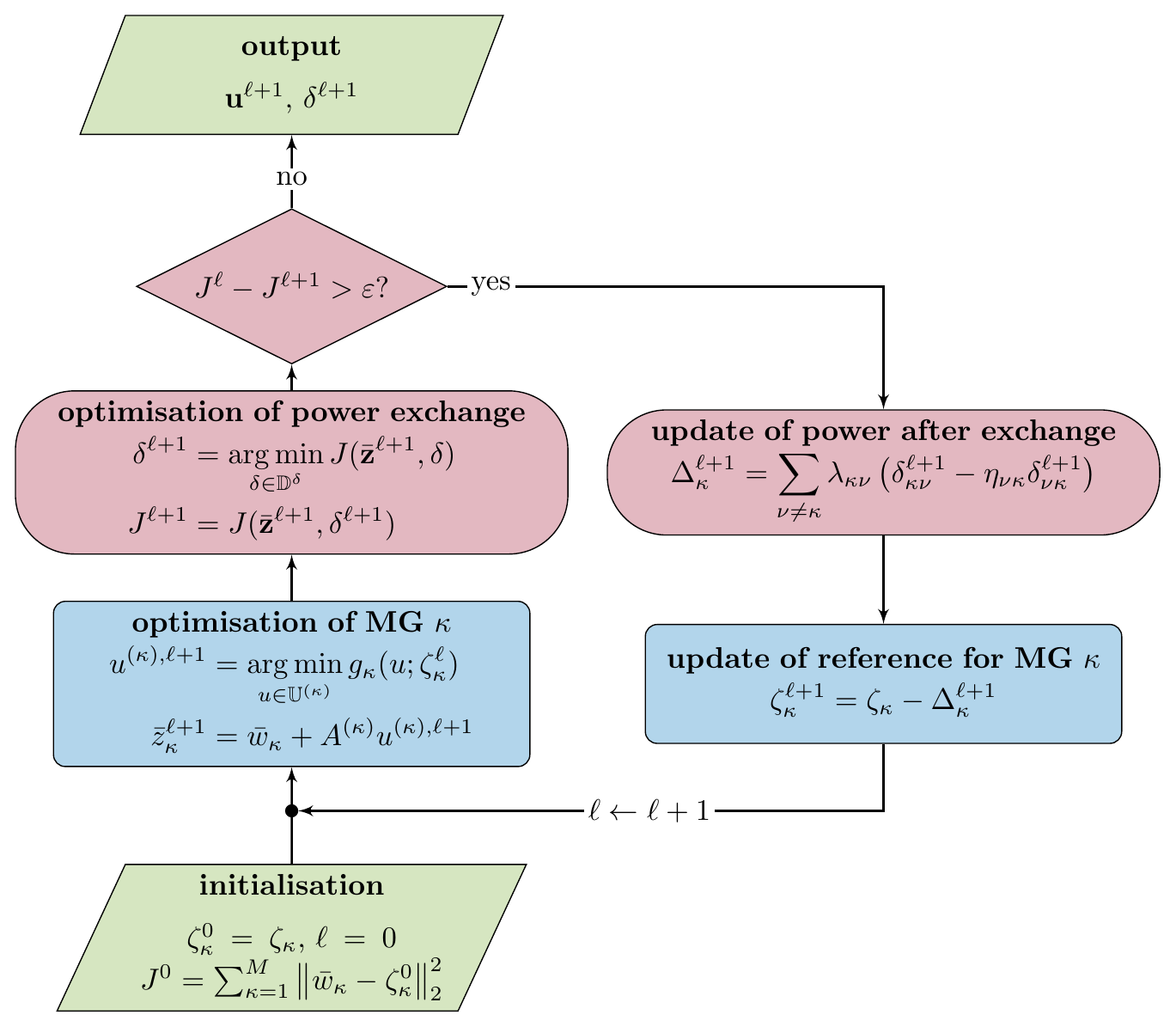}
	\caption{Iterative bidirectional optimisation scheme.\ %
		The green rhomboids mark input and output of the algorithm, while the blue rectangles and red rounded boxes indicate whether the step is performed by each MG in parallel or centrally by the DSO, respectively.}
	\label{fig:bidir_opt_scheme}
\end{figure}
To this end, assume at iteration~$\ell$, $\ell \in \mathbb{N}_0$, each MGO~$\kappa$ has optimised his aggregated power demand~$\bar{z}_\kappa^\ell$ locally by manipulating residential batteries, i.e., OP~\eqref{subeq:OP_lower} has been solved per MG.\ %
Then, OP~\eqref{subeq:OP_upper} is solved to determine the optimal power exchange strategy~$\delta^\ell$.\ %
The idea is to incorporate the exchanged power into the local optimisation problems in order to adjust the battery control.\ %
Let
\begin{align}
	\Delta_\kappa^\ell \; := \; \sum_{\nu \in \mathcal{N}(\kappa)} \lambda_{\kappa \nu} (\delta_{\kappa \nu}^\ell - \eta_{\nu \kappa} \delta_{\nu \kappa}^\ell) \in \mathbb{R}^N \notag
\end{align}
denote the net power that MGO~$\kappa$ receives/provides based on the exchange strategy~$\delta^\ell$.\ %
Then, the overall objective function~$J$ can be written as
\begin{align}
	J(\bar{\mathbf{z}}^\ell, \delta^\ell) \; 
	& = \; \sum_{\kappa=1}^M \left\| \zeta_\kappa - \left[ \bar{z}_\kappa^\ell + \Delta_\kappa^\ell \right] \right\|_2^2 \notag \\
	& = \; \sum_{\kappa=1}^M \left\| A^{(\kappa)} u^{(\kappa)} + b^{(\kappa)} - \zeta_\kappa^\ell \right\|_2^2
	\label{eq:overall_vs_local_objective}
\end{align}
with $\zeta_\kappa^\ell := \zeta_\kappa - \Delta_\kappa^\ell$, $\kappa \in [1:M]$.\ %
Here, we used the notation as in~\eqref{subeq:OP_lower} with the additional superscript~$(\kappa)$ to indicate the dependence on MG~$\kappa$.\ %
Minimising~\eqref{eq:overall_vs_local_objective} with respect to the battery usage can again be parallelised with respect to~$u^{(\kappa)}$.\ %
To this end, we introduce the local objective functions $g_\kappa : \mathbb{R}^{2 N \mathcal{I}_\kappa} \times \mathbb{R}^N \to \mathbb{R}$, 
\begin{align}
	g_\kappa(u^{(\kappa)}; \zeta_\kappa) \; := \; \left\| A^{(\kappa)} u^{(\kappa)} + b^{(\kappa)} - \zeta_\kappa \right\|_2^2, \notag
\end{align}
for all $\kappa \in [1:M]$.\ %
In a next iteration, the local optimisation problems 
\begin{align}
	u^{(\kappa),\ell+1} \; = \; \argmin_{u \in \mathbb{U}^{(\kappa)}} g_\kappa(u; \zeta_\kappa^\ell) \label{eq:OP_lower_update}
\end{align}
are solved in parallel and the procedure is repeated.\ %
The bidirectional optimisation scheme is summarised in Algorithm~\ref{alg:bidir_opt_scheme}.\ %
\begin{algorithm}
	\caption{Iterative bidirectional optimisation scheme}
	{\bf Input}: Current time instance $k \in \mathbb{N}_0$,
	current SoC $x_{\kappa_i}(k) \in \mathbb{X}_{\kappa_i}$, 
	prediction horizon $N \in \mathbb{N}_{\geq 2}$, 
	predicted net consumption $(w_{\kappa, i}(k), \ldots, w_{\kappa, i}(k+N-1))^\top \in \mathbb{R}^N$, $i \in [1:\mathcal{I}_\kappa]$, $\kappa \in [1:M]$, 
	reference trajectories $\zeta_\kappa = (\zeta_\kappa(k), \ldots, \zeta_\kappa(k+N-1))^\top \in \mathbb{R}^N$, $\kappa \in [1:M]$, 
	maximal number $\ell_{\max} \in \mathbb{N}$ of iterations, and 
	tolerance $\varepsilon > 0$. \\
	{\bf Initialisation}:
	\begin{enumerate}
		\item Set $\bar{w}_\kappa = \sum_{i=1}^{\mathcal{I}_\kappa} w_{\kappa, i}$, $\zeta_\kappa^0 = \zeta_\kappa$ for all $\kappa \in [1:M]$, and $\ell = 0$.\ %
		Compute $J^0 = \sum_{\kappa=1}^M \left\| \bar{w}_\kappa - \zeta_\kappa^0 \right\|_2^2$.\ %
		\item \emph{Lower level.}\ %
		Solve $u^{(\kappa),1} = \argmin_{u \in \mathbb{U}^{(\kappa)}} g_\kappa(u; \zeta_\kappa^0)$, compute $\bar{z}_\kappa^1 = \bar{w}_\kappa + A^{(\kappa)} u^{(\kappa),1}$, and send it to the upper level.\ %
		\item \emph{Upper level.}\ %
		Given $\bar{z}_\kappa^1$ for all $\kappa \in [1:M]$, solve~\eqref{subeq:OP_upper} for~$\delta^1$.\ %
		\item Evaluate $J^1 = J(\bar{\mathbf{z}}^1, \delta^1)$.\ %
	\end{enumerate}
	{\bf While} $\ell < \ell_{\max}$ and $J^{\ell} - J^{\ell+1} > \varepsilon$\\ %
	{\bf Do:}
	\begin{enumerate}[resume]
		\item \emph{Upper level.}\ %
		Compute 
		$$\Delta_\kappa^{\ell+1} = \sum_{\nu \neq \kappa} \lambda_{\kappa \nu} (\delta_{\kappa \nu}^{\ell+1} - \eta_{\nu \kappa} \delta_{\nu \kappa}^{\ell+1})$$
		and send it to MG~$\kappa$, $\kappa \in [1:M]$.\ %
		\item \emph{Lower level.}\ %
		Update $\zeta_\kappa^{\ell+1} = \zeta_\kappa - \Delta_\kappa^{\ell+1}$.\ %
		\item Increment $\ell \leftarrow \ell+1$.\ %
		\item \emph{Lower level.}\ %
		Solve $u^{(\kappa),\ell+1} = \argmin_{u \in \mathbb{U}^{(\kappa)}} g_\kappa(u; \zeta_\kappa^\ell)$, compute $\bar{z}_\kappa^{\ell+1} = \bar{w}_\kappa + A^{(\kappa)} u^{(\kappa),\ell+1}$, and send it to the upper level. \ %
		\item \emph{Upper level.}\ %
		Given $\bar{z}_\kappa^{\ell+1}$ for all $\kappa \in [1:M]$, solve~\eqref{subeq:OP_upper} for $\delta^{\ell + 1}$.\ %
		\item \emph{Upper level.}\ %
		Evaluate $J^{\ell+1} = J(\bar{\mathbf{z}}^{\ell+1}, \delta^{\ell+1})$.\ %
	\end{enumerate}
	\label{alg:bidir_opt_scheme}
\end{algorithm}

\subsection{Global Convergence Proof}\label{sec:global_conv_proof}
In this subsection we show that for $\ell \to \infty$ the sequence $(J^\ell)_{\ell \in \mathbb{N}_0}$ generated by Algorithm~\ref{alg:bidir_opt_scheme} converges to the optimal value of~\eqref{eq:OP_full}.\ %
The approach mimics the line of arguments proposed in~\cite[Sec.~IV.B]{BrauGrue16a}.\ %
The key contribution is to make the bi-directional problem accessible to those techniques.\ %

First, we show that Algorithm~\ref{alg:bidir_opt_scheme} successively reduces the overall costs.\ %
\begin{proposition}\label{prop:bidir_mono_decreasing}
	The sequence $(J^\ell)_{\ell \in \mathbb{N}}$ generated by Algorithm~\ref{alg:bidir_opt_scheme} is non-increasing and the corresponding tuple $(\mathbf{z}^\ell, \delta^\ell)$ is feasible for~\eqref{eq:OP_full} for all $\ell \in \mathbb{N}$.\ %
	Furthermore, if $\bar{\mathbf{z}}^{\ell+1} \neq \bar{\mathbf{z}}^\ell$ for some iteration~$\ell$, then $J^{\ell+1} < J^\ell$.\ %
\end{proposition}
\begin{proof}
	Let $(\mathbf{u}^\ell, \bar{\mathbf{z}}^\ell, \delta^\ell, \zeta^{\ell})$ denote the $\ell$-th iterate of Algorithm~\ref{alg:bidir_opt_scheme} after the terminal condition has been checked.\ %
	Note that~$\mathbf{z}^\ell$ and~$\delta^\ell$ are feasible by construction.\ %
	Next, the local optimisation problems~\eqref{eq:OP_lower_update} are solved in parallel.\ %
	Therefore, 
	\begin{align}
		J^\ell 
		\; = \; J(\bar{\mathbf{z}}^\ell, \delta^\ell) 
		\; & = \; \sum_{\kappa=1}^M g_\kappa(u^{(\kappa),\ell}; \zeta_\kappa^\ell) \notag \\
		\; & \geq \; \sum_{\kappa=1}^M g_\kappa(u^{(\kappa), \ell+1}; \zeta_\kappa^\ell)
		\; = \; J(\bar{\mathbf{z}}^{\ell+1}, \delta^\ell). \label{eq:Jell+1<=Jell}
	\end{align}
	Further minimisation of~$J$ with respect to~$\delta$ yields
	\begin{align}
		J(\bar{\mathbf{z}}^{\ell+1}, \delta^\ell) \; \geq \; J(\bar{\mathbf{z}}^{\ell+1}, \delta^{\ell+1}) \; = \; J^{\ell+1} \notag
	\end{align}
	and, hence, $J^\ell \geq J^{\ell+1}$.\ %
	
	Assume $\bar{\mathbf{z}}^{\ell+1} \neq \bar{\mathbf{z}}^\ell$ for some iteration~$\ell$.\ %
	For each~$\kappa$ the set~$\bar{\mathbb{D}}^{(\kappa)}$ is convex and the map $\tilde{g}_\kappa : \mathbb{R}^N \to \mathbb{R}$, $\bar{z} \mapsto \left\| \bar{z} - \zeta_\kappa^\ell \right\|_2^2$, is strictly convex.\ %
	Hence, the minimiser $\bar{z}_\kappa^{\ell+1} = \argmin_{z \in \bar{\mathbb{D}}^{(\kappa)}} \tilde{g}_\kappa(z)$ is unique for all~$\kappa$.\ %
	Therefore, strict inequality holds in~\eqref{eq:Jell+1<=Jell}, which completes the proof.\ %
\end{proof}
The next theorem states that the infimum of the sequence generated by Algorithm~\ref{alg:bidir_opt_scheme} coincides with the optimal value of~\eqref{eq:OP_full}.\ %
In conclusion, Algorithm~\ref{alg:bidir_opt_scheme} converges to a global optimum of~\eqref{eq:OP_full}.\ %
\begin{theorem}
	The infimum~$J^\infty$ of the sequence~$(J^\ell)_{\ell \in \mathbb{N}_0}$ generated by Algorithm~\ref{alg:bidir_opt_scheme} is the optimal value~$J^\star$ of~\eqref{eq:OP_full}.\ %
\end{theorem}
\begin{proof}
	Since by Proposition~\ref{prop:bidir_mono_decreasing} the sequence $(J^\ell)_{\ell \in \mathbb{N}_0}$ generated by Algorithm~\ref{alg:bidir_opt_scheme} is bounded and monotonous, it converges to its infimum.\ %
	Let $(\bar{\mathbf{z}}^\star, \delta^\star)$ denote a (not necessarily unique) optimal solution of~\eqref{eq:OP_full}.\ %
	Then, there exists some $u^\star = (u^{(1),\star}, \ldots, u^{(M), \star}) \in \mathbb{U}$ such that
	\begin{align}
		\bar{z}_\kappa^\star \; = \; \bar{w}_\kappa + A^{(\kappa)} u^{(\kappa), \star} \notag
	\end{align}
	holds for all $\kappa \in [1:M]$.\ %
	Furthermore, assume an iterate $(\bar{\mathbf{z}}^\ell, \delta^\ell)$ of Algorithm~\ref{alg:bidir_opt_scheme} to be given, which is not optimal, i.e., there exists some $u^\ell \in \mathbb{U}$ such that
	\begin{align}
		\bar{z}_\kappa^\ell \; = \; \bar{w}_\kappa + A^{(\kappa)} u^{(\kappa), \ell} \notag
	\end{align}
	for all $\kappa \in [1:M]$ and
	\begin{align}
		J^\star \; = \; J(\bar{\mathbf{z}}^\star, \delta^\star) \; < \; J(\bar{\mathbf{z}}^\ell, \delta^\ell) \; = \; J^\ell. \label{eq:bidir_convergence_proof_Jell_not_optimal}
	\end{align}
	Then, the reference~$\zeta^\ell$ is updated in Steps~4 and~5 of Algorithm~\ref{alg:bidir_opt_scheme} based on the power exchange~$\delta^\ell$.\ %
	Next, Algorithm~\ref{alg:bidir_opt_scheme} computes an update~$(\bar{\mathbf{z}}^{\ell+1}, \delta^{\ell+1})$ such that
	\begin{align}
		J^{\ell+1} \; = \; J(\bar{\mathbf{z}}^{\ell+1}, \delta^{\ell+1}) \; \leq \; J(\bar{\mathbf{z}}^\ell, \delta^\ell) \; = \; J^\ell. 
		\label{eq:bidir_convergence_proof_descent_condition}
	\end{align}
	
	Since, $J$ is convex (on a convex domain) and differentiable, its derivative at $(\bar{\mathbf{z}}^\ell, \delta^\ell)$ in direction of the optimum $(\bar{\mathbf{z}}^\star, \delta^\star)$ is negative, i.e.,
	\begin{align}
		0 \; & > \; \left( \nabla J(\bar{\mathbf{z}}^\ell, \delta^\ell) \right)^\top \left( \begin{pmatrix} \bar{\mathbf{z}}^\star \\ \delta^\star \end{pmatrix} - \begin{pmatrix} \bar{\mathbf{z}}^\ell \\ \delta^\ell \end{pmatrix} \right) \; \notag \\
		& = \; \left( \nabla_{\bar{\mathbf{z}}} J(\bar{\mathbf{z}}^\ell, \delta^\ell) \right)^\top (\bar{\mathbf{z}}^\star - \bar{\mathbf{z}}^\ell) + \left( \nabla_{\delta} J(\bar{\mathbf{z}}^\ell, \delta^\ell) \right)^\top (\delta^\star - \delta^\ell). \notag
	\end{align}
	If $\left( \nabla_{\bar{\mathbf{z}}} J(\bar{\mathbf{z}}^\ell, \delta^\ell) \right)^\top (\bar{\mathbf{z}}^\star - \bar{\mathbf{z}}^\ell) < 0$, then strict inequality holds in~\eqref{eq:bidir_convergence_proof_descent_condition} since optimising the batteries yields a strict improvement.\ %
	Now, assume $\left( \nabla_{\bar{\mathbf{z}}} J(\bar{\mathbf{z}}^\ell, \delta^\ell) \right)^\top (\bar{\mathbf{z}}^\star - \bar{\mathbf{z}}^\ell) = 0$.\ %
	Then, $\bar{\mathbf{z}}^\ell = \bar{\mathbf{z}}^\star$ due to strict convexity of~$J$ with respect to~$\bar{\mathbf{z}}$.\ %
	Therefore, 
	\begin{align}
		\delta^\ell \in \argmin_{\delta \in \mathbb{D}^\delta} J(\bar{\mathbf{z}}^\ell, \delta) = \argmin_{\delta \in \mathbb{D}^\delta} J(\bar{\mathbf{z}}^\star, \delta), \notag 
	\end{align}
	which means that $(\bar{\mathbf{z}}^\ell, \delta^\ell)$ is optimal, in contradiction to~\eqref{eq:bidir_convergence_proof_Jell_not_optimal}.\ %
	Consequently, strict inequality holds in~\eqref{eq:bidir_convergence_proof_descent_condition}.\ %
	
	Since the objective function~$J$ is continuous and the feasible set $\bar{\mathbb{D}} \times \mathbb{D}^\delta$ is compact, there exists an accumulation point $(\hat{\mathbf{z}}, \hat{\delta}) \in \bar{\mathbb{D}} \times \mathbb{D}^\delta$ of the sequence $(\bar{\mathbf{z}}^\ell, \delta^\ell)_{\ell \in \mathbb{N}_0}$ constructed by Algorithm~\ref{alg:bidir_opt_scheme} such that
	\begin{align}
		\hat{J} \; := \; J(\hat{\mathbf{z}}, \hat{\delta}) \; = \; J^\infty. \notag
	\end{align}
	Clearly, $\hat{J} \geq J^\star$.\ %
	We assume $\hat{J} > J^\star$ and derive a contradiction.\ %
	The strict inequality in~\eqref{eq:bidir_convergence_proof_descent_condition} at the accumulation point $(\hat{\mathbf{z}}, \hat{\delta})$ in combination with the continuity of~$J$ implies that
	\begin{align}
		J(\bar{\mathbf{z}}^{\ell+1}, \delta^{\ell+1}) \; < \; \hat{J}
	\end{align}
	for all $(\bar{\mathbf{z}}^\ell, \delta^\ell) \in \mathcal{B}_\varepsilon(\hat{\mathbf{z}}, \hat{\delta})$ for sufficiently small $\varepsilon > 0$.\ %
	However, according to Proposition~\ref{prop:bidir_mono_decreasing} the sequence $(J^\ell)_{\ell \in \mathbb{N}_0}$ is monotonically non-increasing which contradicts the definition of the accumulation point $(\hat{\mathbf{z}}, \hat{\delta})$.\ %
	Therefore, $J^\infty = J^\star$, which completes the proof.\ %
\end{proof}

\section{Numerical Case Study}\label{sec:numerics}
In section, we illustrate the potential of the proposed approach in a numerical case study based on real-world data.\ %

\subsection{Implementation Details}
We use an alternating direction method of multipliers (ADMM)~\cite{BoydPari11} to solve~\eqref{eq:OP_lower_update} in a distributed manner within each MG in parallel.\ %
Furthermore, the optimisation with respect to the power exchange can be decoupled in time.\ %
Thus, the objective function for the power exchange during one time step~$n$ reduces to 
\begin{align}
	& J_n(\bar{\mathbf{z}}(n), \delta(n)) \notag \\
	= &  \sum_{\kappa=1}^M \left( y_\kappa(n) - a_\kappa^\top \delta(n) \right)^2 \notag \\
	= &  \delta(n)^\top \sum_{\kappa=1}^M (a_\kappa a_\kappa^\top) \delta(n) -2 \sum_{\kappa=1}^M (y_\kappa(n) a_\kappa^\top) \delta(n), \notag 
\end{align}
where
\begin{align}
	a_1 \; & = \; (\lambda_{12}, -\eta_{21} \lambda_{21}, \lambda_{13}, -\eta_{31} \lambda_{31}, \lambda_{14}, -\eta_{41} \lambda_{41}, 0, 0)^\top \notag \\
	a_2 \; & = \; ( -\eta_{12} \lambda_{12}, \lambda_{21}, 0, 0, 0, 0, \lambda_{24}, -\eta_{42} \lambda_{42} )^\top \notag \\
	a_3 \; & = \; ( 0, 0, -\eta_{13} \lambda_{13}, \lambda_{31}, 0, 0, 0, 0 )^\top \notag \\ 
	a_4 \; & = \; ( 0, 0, 0, 0, -\eta_{14} \lambda_{14}, \lambda_{41}, -\eta_{24} \lambda_{24}, \lambda_{42} )^\top \notag 
\end{align} 
satisfying $a_\kappa \delta(n) = \Delta_\kappa(n)$ encode the grid topology visualised in Figure~\ref{fig:exchange_structure} and with the notions $y_\kappa(n) = \zeta_\kappa(n) - \bar z_\kappa(n)$ and 
\begin{align}
	\delta(n) \; := \; ( \delta_{12}, \delta_{21}, \delta_{13}, \delta_{31}, \delta_{14}, \delta_{41}, \delta_{24}, \delta_{42} )^\top (n) \in \mathbb{R}^8. \notag 
\end{align}
In other words, we replace the large-scale optimisation problem $\min_{\delta \in \mathbb{D}^\delta} J(\bar{\mathbf{z}}, \delta)$ by $N$ parallelisable QPs.\ %
Thus, the main computational effort lies with the local optimisation of the MGs which can be done efficiently using distributed optimisation techniques as discussed, e.g.\ in~\cite{BrauFaul18,JianSaue21}.\ %
Alternatively, one might use surrogate models to approximate the solution of the lower level optimisation and, thus, reduce computation and communication effort as suggested in the underlying work~\cite{BaumGrun19}.\ %
The QPs are solved using the \texttt{MATLAB}-inherent toolbox \texttt{quadprog}.\ %

\subsection{Description of Data used in Simulations}
The load profiles~$\ell_i$ as visualised in Figure~\ref{fig:load_profiles} are taken from the ISSDA data set\footnote{\url{https://www.ucd.ie/issda/data/commissionforenergyregulationcer/}}.\ %
\begin{figure}[h]
	\includegraphics[width=\columnwidth]{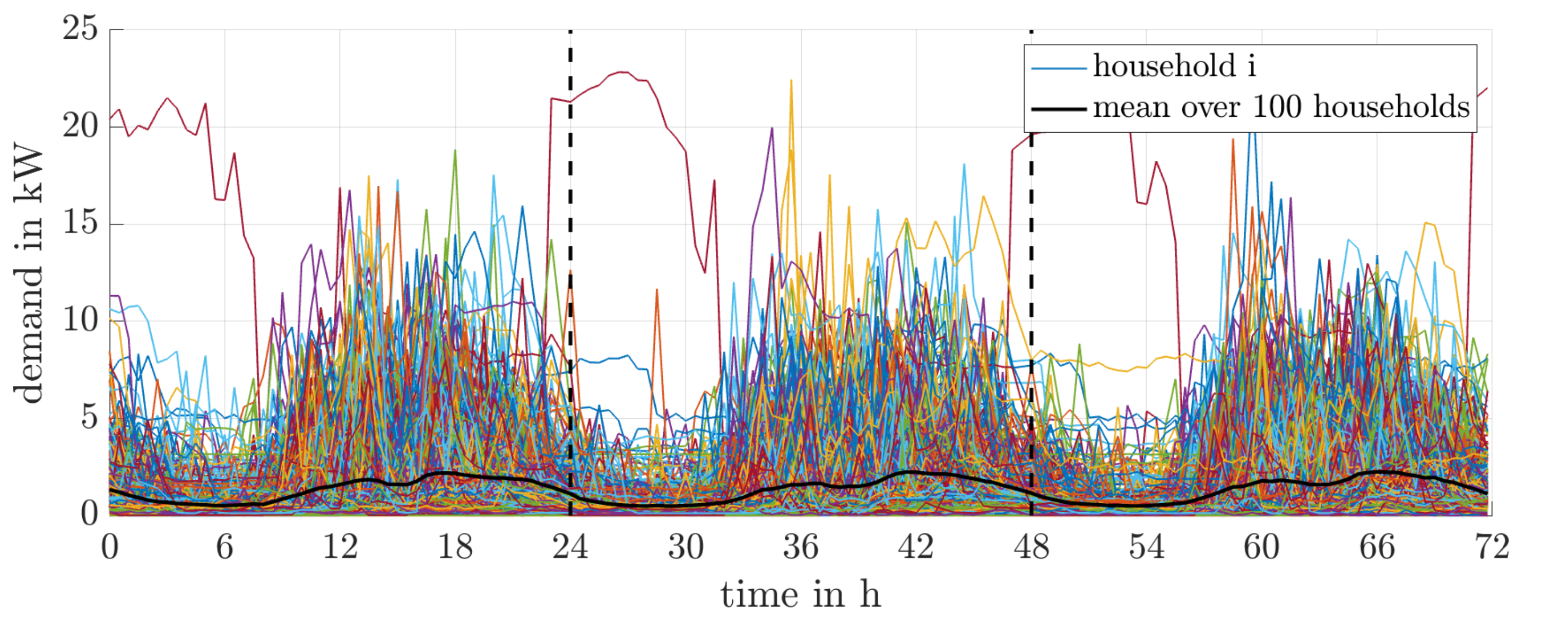}
	\caption{Load profiles of 100 residential homes and their average over three days.\ %
		The dashed vertical lines indicate the time window considered in the simulations.}
	\label{fig:load_profiles}
\end{figure}
In our simulations we used in total 100 households for the four MGs depicted in Figure~\ref{fig:exchange_structure} and consider a time span of one day as highlighted by the dashed black lines in Figure~\ref{fig:load_profiles} (plus prediction of the consecutive day).\ %
The data is given in 15min intervals, thus, $T = 0.25$.\ %

For the residential generation~$g_i$ we used the hourly radiation data from the EU Science Hub\footnote{\url{https://ec.europa.eu/jrc/en/pvgis}}.\ %
Here, only one profile is provided.\ %
In order to use it in our simulations, we simply added the resulting generated power every single household.\ 
In addition to the \emph{actual} generation, a (virtual) \emph{predicted} generation profile for each household was created.\footnote{\url{https://github.com/klaus-rheinberger/DSM-data}}\ %
Examples are visualized in Figure~\ref{fig:pv_profiles}.\ 
\begin{figure}[h]
	\centering
	\includegraphics[width=\columnwidth]{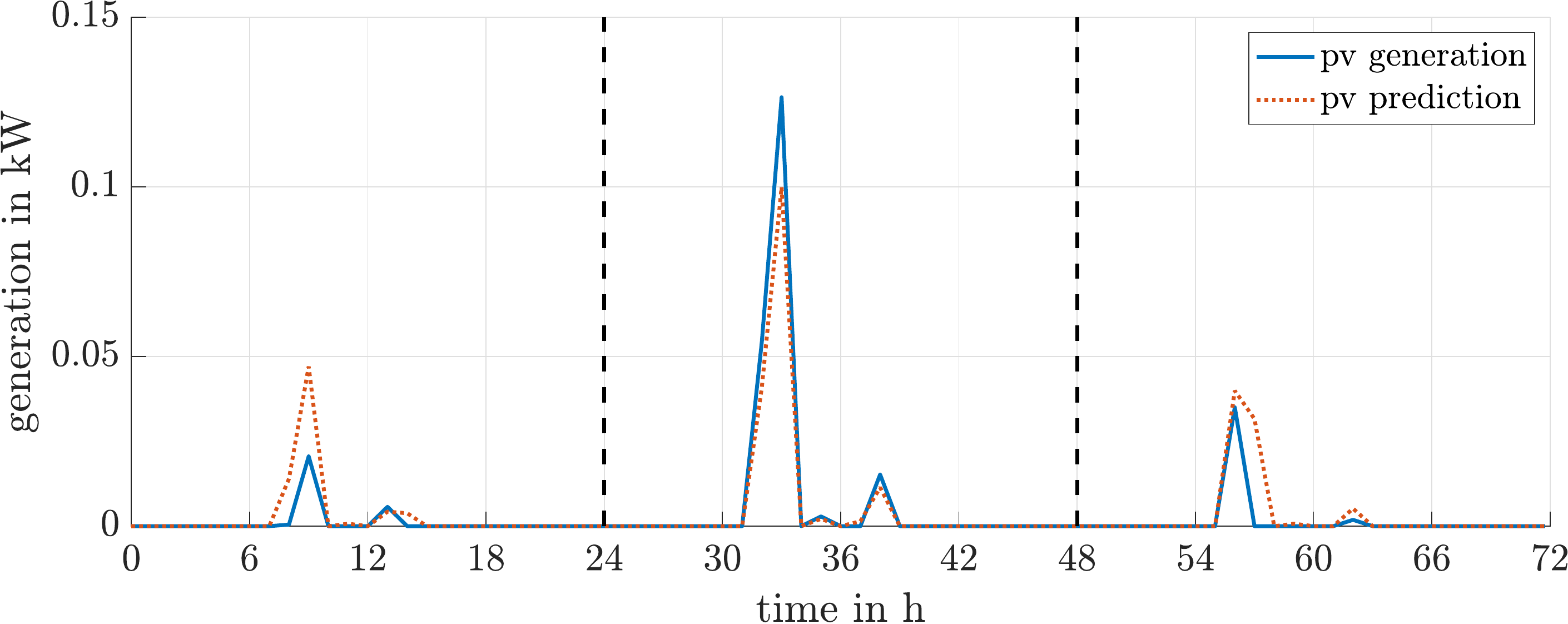}
	\caption{Predicted versus actual photovoltaic (PV) generation profile over three days.\ %
		The dashed vertical lines indicate the time window considered in the simulations.}
	\label{fig:pv_profiles}
\end{figure}
Note, however, that the scale of the generation is insignificantly smaller than the scale of the load.\ %

We further equipped each household virtually with a battery with randomly chosen parameters.\ %
The corresponding random distributions are 
\begin{align}
	C_{\kappa,i} & \sim \mathcal{N}(2,0.5) \text{ [kWh]} \quad x_{\kappa, i} \sim \mathcal{N}(0.5, 0.05) \text{ [kWh]} \notag \\
	\alpha_{\kappa,i} & \sim \mathcal{N}(0.99, 0.01) \quad \beta_{\kappa, i}, \gamma_{\kappa, i} \sim \mathcal{N}(0.95, 0.05) \notag \\
	- \underline{u}_{\kappa, i}, \bar{u}_{\kappa, i} & \sim \mathcal{N}(0.25, 0.15) \cdot C_{\kappa,i} \text{ [kW]}, \notag 
\end{align} 
where $\mathcal{N}(\mu, \sigma)$ denotes the normal distribution with expected value~$\mu$ and standard deviation~$\sigma$.\ %
The prediction horizon length was chosen as $N = 96$ (one day) while the line efficiencies and the line limits are given by
\begin{align}
	\eta \; = \; 
	\begin{bmatrix}
		0    & 0.8   &   0.9 &  0.85 \\
		0.8   &    0   &   0    &  0.9 \\
		0.9   &   0    &   0    &    0 \\
		0.85 &   0.9  &  0    &    0
	\end{bmatrix}
	\; \text{and} \; 
	\lambda \; = \; 
	\begin{bmatrix}
		0   &  9  &   8  &   7 \\
		9   &  0  &   0  &   8 \\
		8   &  0  &   0  &   0 \\
		7   &  8  &   0  &   0 
	\end{bmatrix}. 
	\notag 
\end{align}

\begin{remark}
	Instead of introducing different reference trajectories~$\zeta_\kappa$ for each MG~$\kappa$, one could use a uniform $\zeta$ for all MGs.\ %
	In our simulations, the goal of the power exchange is to further reduce the overall peak shaving.\ %
	To this end, we use the share 
	\begin{align}
		\zeta_\kappa(n) \; = \; \frac{\mathcal{I}_\kappa}{\sum_{\nu=1}^M \mathcal{I}_\nu} \; \zeta(n) \quad \forall \, n \in \{k, \ldots, k+N-1\}, \notag 
	\end{align}
	of the overall average net consumption 
	\begin{align}
		\zeta(n) \; = \; \frac 1 N \sum_{j=n-N+1}^n \sum_{\kappa=1}^M \sum_{i=1}^{\mathcal{I}_\kappa} w_{\kappa, i}^\mathrm{pred}(j) \notag 
	\end{align}
	over the last~$N$ time steps (for $n \geq N$).\ %
\end{remark}

\subsection{Open-Loop Results}
The improvement of the overall performance based on Algorithm~\ref{alg:bidir_opt_scheme} (for one time step in open loop) is depicted in Figures~\ref{fig:convergence}.\ %
\begin{figure}[h]
	\centering
	\includegraphics[width=\columnwidth]{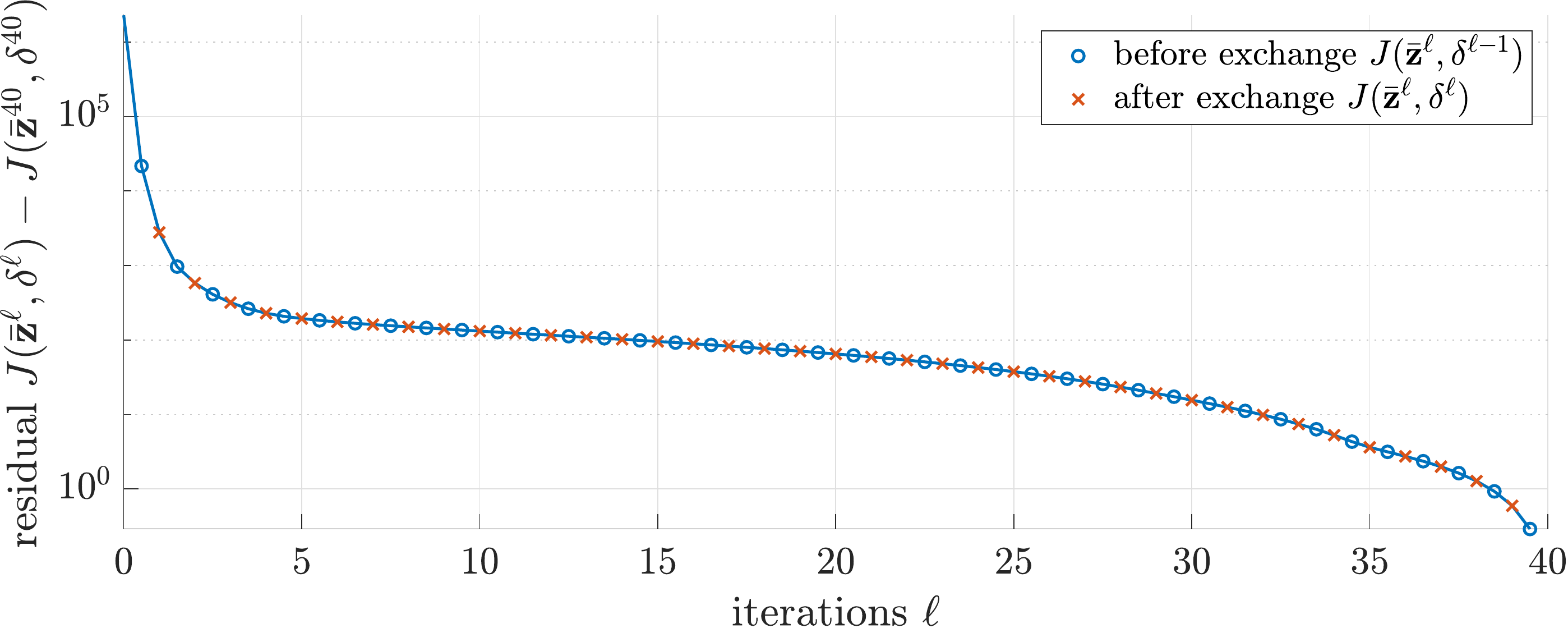}
	\caption{Convergence of Algorithm~\ref{alg:bidir_opt_scheme} for the first 40 iterations.}
	\label{fig:convergence}
\end{figure}%
Here, we considered four MGs with topology as depicted in Figure~\ref{fig:exchange_structure}.\ %
The MGs consist of $\mathcal{I}_\kappa = 40, \, 20, \, 20, \, 20$ residential units, respectively.\ %
The corresponding evolution of the objective function values is listed in Table~\ref{tab:obj_values}.\ %
\begin{table}[h]
	\centering
	\caption{Evolution of the objective function value $J(\bar{\mathbf{z}}^\ell, \delta^{\ell-1})$ (before exchange) and $J(\bar{\mathbf{z}}^\ell, \delta^\ell)$ (after exchange) generated by Algorithm~\ref{alg:bidir_opt_scheme}}
	\begin{tabular}{c|rr}
		iteration & before exchange & after exchange \\
		\hline 
		0 & 2,313,011 & -- \\
		1 & 51,602 & 32,727 \\
		2 & 30,915 & 30,529 \\ 
		3 & 30,358 & 30,266 \\
		\vdots & \vdots & \vdots \\
		39 & 29,949 & 29,949 \\
		40 & 29,949 & 29,948
	\end{tabular}
	\label{tab:obj_values}
\end{table}
Note that main reduction of the objective function value is achieved within the first step, i.e.\ by controlling the local batteries.\ %
However, additional power exchange reduces the costs further by more than 40\%.\ %
Moreover, the overall objective function value could be slightly decreased further by running Algorithm~\ref{alg:bidir_opt_scheme} for more iterations.\ %
However, we stopped after 40 iterations since the relative improvement 
\begin{align}
	\frac{J(\bar{\mathbf{z}}^{40}, \delta^{39}) - J(\bar{\mathbf{z}}^{40}, \delta^{40})}{J(\bar{\mathbf{z}}^{40}, \delta^{39})} \; & \approx \; \frac{29948.60 - 29948.31}{29948.60} \notag \\
	& \approx \; 9.6833 \cdot 10^{-6} \notag 
\end{align}
was sufficiently small.\ %
\begin{remark}
	The purpose of the simulations is a proof of concept; this case study does not show the full potential of the power exchange or the presented optimisation approach.\ %
	The data is publicly available and has not been further manipulated by the authors.\ %
	In particular, the average load and generation profiles within each MG are qualitatively similar.\ %
	In practice, it is often the case that, e.g.\ due to generation via renewables, there is a power surplus in one and a demand in another region.\ %
	In such scenarios, power exchange might yield even better performances.\ %
\end{remark}

\subsection{Closed-Loop Results}
In practice, one would not implement the prediction-based solution at once but rather update it iteratively once new data --~in our case weather forecast and, thus, generation prediction~-- comes available.\ %
This procedure is typically referred to as model predictive control (MPC) and is summarised in Algorithm~\ref{alg:mpc}.\ %
\begin{algorithm}[h]
	\caption{Model predictive control for solving~\eqref{eq:OP_full}}
	{\bf Input}: prediction horizon length $N \in \mathbb{N}_{\geq 2}$,  
	desired reference trajectories $\zeta_\kappa = (\zeta_\kappa(k), \ldots, \zeta_\kappa(k+N-1))^\top \in \mathbb{R}^N$, $\kappa \in [1:M]$. \\
	{\bf Initialise}: $k = 0$. \\
	{\bf Repeat}: 
	\begin{enumerate}
		\item Measure current SoC $x_{\kappa_i}(k) \in \mathbb{X}_{\kappa_i}$ and predict net consumption $w^\mathrm{pred} = (w^\mathrm{pred}(k), \ldots, w^\mathrm{pred}(k+N-1))^\top$.
		\item Run Algorithm~\ref{alg:bidir_opt_scheme} to obtain optimal control $u_k^{\kappa,\star} = (u_k^{\kappa, \star}(k), \ldots, u_k^{\kappa,\star}(k+N-1))^\top \in \mathbb{R}^{2N}$ for each MG~$\kappa$ and $\delta_k^\star = (\delta_k(k)^\top, \ldots, \delta_k(k+N-1)^\top) \in \mathbb{R}^{2|E|N}$.
		\item Implement the first control impulse $\mu_k := (u_k^{\kappa,\star}(k)), \delta_k^\star(k))$. 
		\item Shift time window and increment $k \leftarrow k+1$.
	\end{enumerate}
	{\bf Output}: control sequence $(\mu_k)_{k \in \mathbb{N}_0}$.
	\label{alg:mpc}
\end{algorithm}
For an introduction to MPC we refer to~\cite{RawlMayn17}.\ %

A comparison of the open-loop performance and the MPC closed loop is given in Figures~\ref{fig:overall_OL_CL} and~\ref{fig:single_MGs_OL_CL}.\ %
\begin{figure}[h]
	\centering
	\includegraphics[width=\columnwidth]{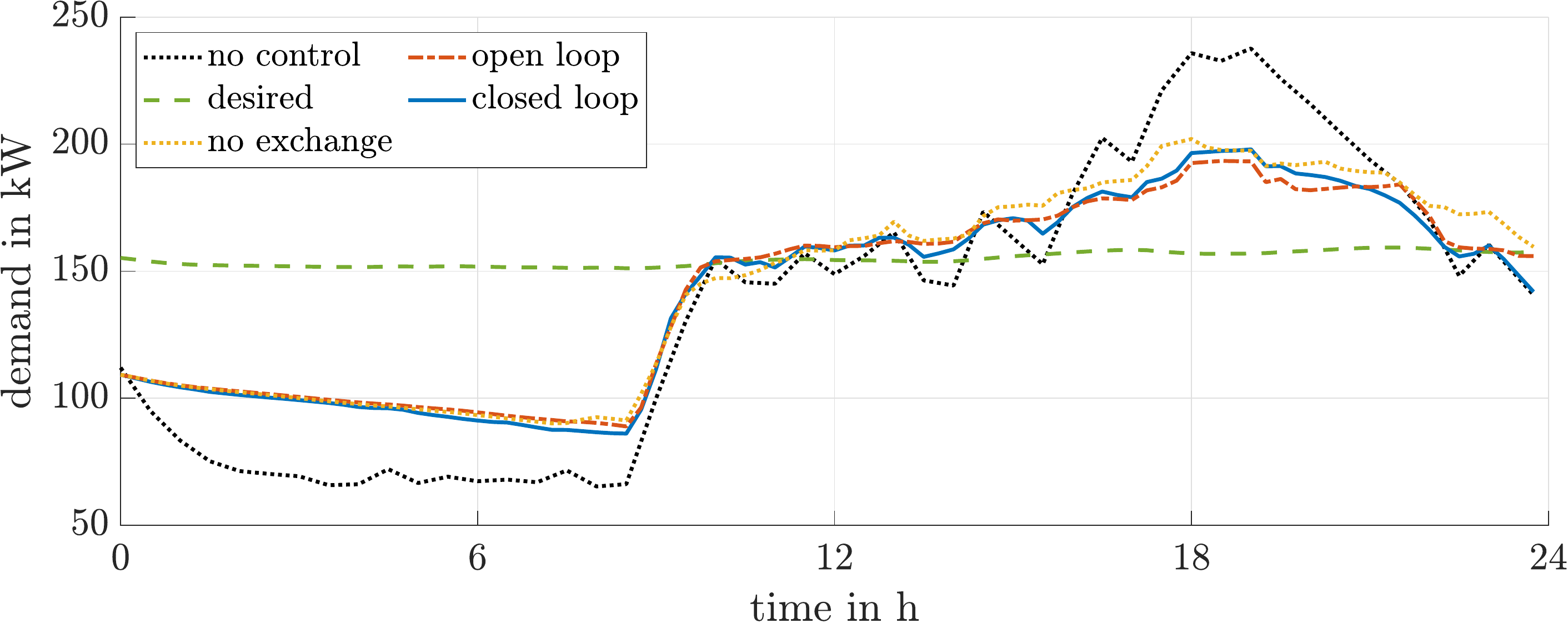}
	\caption{Open versus closed-loop performance.}
	\label{fig:overall_OL_CL}
\end{figure}%
\begin{figure}[h]
	\centering
	\includegraphics[width=\columnwidth]{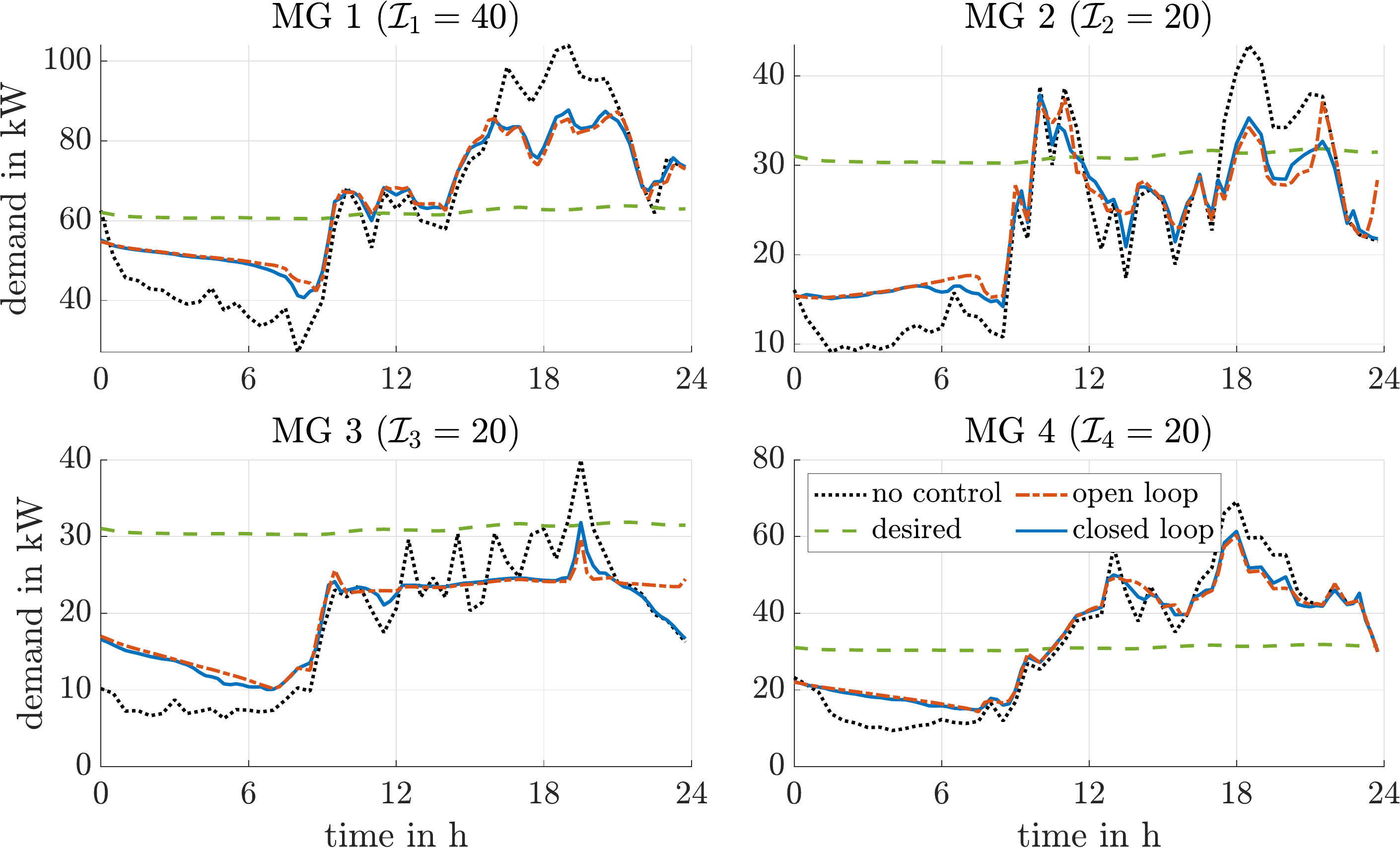}
	\caption{Open versus closed-loop performance within each MG.}
	\label{fig:single_MGs_OL_CL}
\end{figure}%
In Figure~\ref{fig:overall_OL_CL}, additionally, the power demand after optimising the batteries but before exchanging power is visualised.\ %
As mentioned above, the impact of further exchanging power is rather small but still significant.\ %
The closed-loop solution is close to the one in open loop.\ %
However, in particular, towards the end of the considered time window the performance is slightly worse.\ %
This is a typical phenomenon of MPC since it takes predictions for the consecutive day into account and, hence, adjusts the control accordingly.\ %

\section{Conclusions}\label{sec:conclusions}
In this paper, we proposed a bidirectional optimisation scheme for the hierarchical optimisation of partially interconnected microgrids (MGs).\ %
While in the previously published model~\cite{BaumGrun19} the power exchange among neighbouring MGs was limited by the power demand and, thus, no power could be exchanged if the demand was zero, we circumvent this problem via an improved model formulation.\ %
Thus, more flexibility within the grid is available reducing the overall load shaping costs.\ %
We proved global convergence of the optimisation scheme exploiting the convexity of the novel problem formulation and demonstrated the potential in a numerical case study based on real-world data.\ %

It is straightforward to combine the presented approach with the use of surrogate models as suggested in~\cite{BaumGrun19} in order to further reduce the communication overhead as well as the computational effort.\ %
Another possible extension is to consider more than two layers of the grid hierarchy and apply the presented approach at each interface connecting two layers.\ %

\bibliographystyle{plain}
\bibliography{references.bib}

\end{document}